\documentclass[11pt,reqno]{amsart}
\usepackage[a4paper,margin=30mm]{geometry}
\usepackage{amsmath,amssymb,mathtools}
\usepackage{microtype}
\usepackage[colorlinks=true,linkcolor=blue,citecolor=blue,urlcolor=blue]{hyperref}
\hypersetup{pdftitle={Counterexamples to prescribed scalar curvature in the null case},pdfauthor={}}

\newtheorem{theorem}{Theorem}[section]
\newtheorem{corollary}[theorem]{Corollary}
\newtheorem{lemma}[theorem]{Lemma}
\theoremstyle{remark}
\newtheorem{remark}[theorem]{Remark}
\newcommand{\T}{\mathbb T}
\newcommand{\R}{\mathbb R}
\newcommand{\Z}{\mathbb Z}
\newcommand{\dd}{\,dx}
\newcommand{\av}{\operatorname{avg}}
\newcommand{\diver}{\operatorname{div}}
\newcommand{\Hess}{\operatorname{Hess}}
\newcommand{\sym}{\operatorname{sym}}

\numberwithin{equation}{section}

\author[Y. Li, Y. Wu, J. Zhou]{Yuxiang Li, Yunqing Wu, Jie Zhou}
\address{
	Department of Mathematical Sciences\\
	Tsinghua University, Beijing, People's Republic of China}
\email{liyuxiang@tsinghua.edu.cn}
\address{
	The Institute of Geometry and Physics\\
	University of Science and Technology of China, Hefei, Anhui, People's Republic of China}
\email{yqwu19@ustc.edu.cn}
\address{
	School of Mathematical Sciences \\
	Capital Normal University, Beijing, People's Republic of China}
\email{zhoujiemath@cnu.edu.cn}
\title[Prescribed scalar curvature in the null case]{Counterexamples to prescribed scalar curvature in the null case}
\date{}

\begin{document}
\begin{abstract}
We construct smooth functions on  the flat torus $\T^{n\ge5}$  that change sign and have negative mean but cannot be realized as the scalar curvature of any metric in the given conformal class.
\end{abstract}
\maketitle

\section{Introduction}

Let $(M^{n\ge 3},g)$ be a closed connected Riemannian manifold  and  $\Delta_g=\diver_g\nabla$. A conformal metric $\widetilde g=w^{4/(n-2)}g$  has scalar curvature $f$ if and only if 
\begin{equation}\label{eq:conformal}
 L_gw=fw^p,\qquad L_g=-c_n\Delta_g+R_g,\qquad
 c_n=\frac{4(n-1)}{n-2},\quad p=\frac{n+2}{n-2}.
\end{equation}
The sign of the first eigenvalue of the conformal Laplacian $L_g$ is an invariant of the conformal class and separates the problem into positive, null, and negative cases. In the null case, a positive first eigenfunction supplies a scalar-flat representative of the conformal class. We fix such a representative $g$, so $R_g\equiv0$. After setting $u=c_n^{-1/(p-1)}w$, equation \eqref{eq:conformal} becomes
\begin{equation}\label{eq:normalized}
 -\Delta_gu=fu^p,\qquad u>0.
\end{equation}

Two necessary conditions follow immediately. Integration of \eqref{eq:normalized} gives $\int_Mfu^p\,dV_g=0$, so every nonzero prescribed function must change sign. Multiplication by $u^{-p}$ and integration by parts give
\begin{equation}\label{eq:negative-mean}
 \int_M f\,dV_g
 =-p\int_Mu^{-p-1}|\nabla u|_g^2\,dV_g<0
 \quad\text{if }f\not\equiv0.
\end{equation}
Thus the natural candidate characterization in the null case is
\begin{equation}\label{eq:conditions}
 f\equiv0,\qquad\text{or}\qquad
 f\text{ changes sign and }\int_M f\,dV_g<0.
\end{equation}
 The null-case prescribed scalar curvature question considered here asks whether these necessary conditions are also sufficient.

In low dimensions, these conditions give an existence criterion. For a closed surface with zero Euler characteristic, Kazdan and Warner proved the corresponding statement for Gaussian curvature in a fixed conformal class, expressed in a flat representative \cite{KW74}. In dimensions three and four, Escobar and Schoen proved that, on a closed scalar-flat background, every nonzero smooth function satisfying \eqref{eq:conditions} is the scalar curvature of a pointwise conformal metric \cite{ES86}. The present paper shows that these two elementary conditions do not remain sufficient in dimensions $n\geq5$, already in the standard flat conformal class on a torus.

Let $\T^n=\R^n/\Z^n$ be equipped with the standard flat metric $g_0$, so that its volume is one. Fix $x_0\in\T^n$ and an embedded Euclidean ball $B_r(x_0)$. A function $\phi\in C^\infty(\T^n)$ will be called admissible if
\begin{equation}\label{eq:admissible}
 \phi\geq0,\qquad \phi^{-1}(0)=\{x_0\},\qquad
 \phi(x)=|x-x_0|^2\quad\text{on }B_r(x_0).
\end{equation}

\begin{theorem}\label{thm:main}
Let $n\geq5$ and let $\phi$ be admissible. There exists $A_0>0$, depending only on the fixed torus and $\phi$, such that, for every $A\geq A_0$, the function
\[
 f_A=1-A\phi
\]
changes sign, has negative integral, and the equation
\begin{equation}\label{eq:main-pde}
 -\Delta_{g_0}u=f_Au^{(n+2)/(n-2)}
\end{equation}
has no positive smooth solution.
\end{theorem}

\begin{corollary}\label{cor:geometry}
For every $n\geq5$, there are smooth functions on $\T^n$ that change sign and have negative integral but are not the scalar curvature of any smooth metric pointwise conformal to $g_0$.
\end{corollary}

 For all sufficiently large $A$, the positive set of $f_A$ is exactly $B_{A^{-1/2}}(x_0)$, while $\max f_A=1$. The examples therefore have one localized positive region and a negative background.
The proof uses estimates obtained directly from the equation. The integrated equation controls the nonlinear mass outside a fixed ball. Subharmonic estimates and the Green representation then control the exterior values and gradients of any hypothetical solution. Applying a Pohozaev identity to $X=\nabla\phi$ gives two bounds for a weighted moment which contradicts for large $A$. The exact squared-distance form of $\phi$ makes the error coefficients vanish in the ball where concentration could occur. The dimension threshold enters through
\[
 \int_{B_r(x_0)}|x-x_0|^{-2p}\dd<\infty
 \quad\Longleftrightarrow\quad n>2p,
\]
which, at the critical exponent, holds for integer dimensions precisely when $n\geq5$. No bound on the total energy and no description of possible concentration points is used. The estimates apply uniformly to every hypothetical positive solution.

Section~\ref{sec:estimates} establishes the integral identity and the estimates. Section~\ref{sec:proof} proves the theorem and its geometric corollary. 

\vspace{2ex}
\noindent\textbf{AI assistance.}
The counterexample presented in this paper was suggested by  ChatGPT. The authors have verified its correctness and take full responsibility for the mathematical content of the paper.

\section{Estimates for a hypothetical solution}\label{sec:estimates}

For simplicity, we denote $B_r(x_0)$ by $B_r$ and identify it with the Euclidean ball of radius $r$ centered at the origin. All integrals in this section use the volume form of $g_0$, denoted by $dx$, and $\Delta=\Delta_{g_0}$. Put
\[
 p=\frac{n+2}{n-2},\qquad q=p+1=\frac{2n}{n-2}.
\]
Throughout this section, constants $C$ depend only on the fixed torus, $n$, $\phi$, $r$, and fixed auxiliary cutoffs. They never depend on $A$ or on the hypothetical solution. Their values may change from line to line. Suppose that $A$ is sufficiently large and $u>0$ is a smooth solution of \eqref{eq:main-pde}.

The classical Pohozaev identity is obtained by  testing the equation by the inner variation $X_0\cdot\nabla u$ for $X_0=\frac{1}{2}\nabla |x-x_0|^2$ or the exterior variation $u^p$. The following combines both. 
\begin{lemma}[Pohozaev identity]\label{lem:identity}
Let $n\geq3$, $f\in C^\infty(\T^n)$, and let $u>0$ solve $-\Delta u=fu^p$. For every smooth vector field $X$ on the torus,
\begin{align}\label{eq:vector-identity}
 -\frac1q\int_{\T^n}(Xf)u^q\dd
 =\int_{\T^n}\left[\sym(DX)-\frac{\diver X}{n}g_0\right](\nabla u,\nabla u)\dd -\frac1{2q}\int_{\T^n}\Delta(\diver X)u^2\dd.
\end{align}
\end{lemma}

\begin{proof}  For any smooth vector field $X$,  we set $s=\diver X$ and test the equation by the inner variational field $X\cdot \nabla u$.  Integration by parts gives 
\begin{align*}
\int(-\Delta u)(X\cdot\nabla u)\dd
 &=\int\sym(DX)(\nabla u,\nabla u)\dd
   -\frac12\int s|\nabla u|^2\dd,
\end{align*}
and
\begin{align*}
 \int fu^p(X\cdot\nabla u)\dd
 =-\frac1q\int(Xf+fs)u^q\dd.
\end{align*}
There occurs an exterior variation term $\int fsu^qdx=\int fu^p\cdot su dx$. To cancel this term, test the equation by  the exterior variation field $su$. Again integration by parts implies 
\begin{align*}
 \int fsu^qdx=\int(-\Delta u)su\dd=\int s|\nabla u|^2\dd-\frac12\int(\Delta s)u^2\dd.
\end{align*}
Adding the second identity with coefficient $1/q$ to the first
and using $1/2-1/q=1/n$ yields \eqref{eq:vector-identity}.
\end{proof}

In the setting $f=f_A$, for $X=\nabla\phi$, there holds $Xf_A=-A|\nabla\phi|^2$ and
\[
\sym(DX)-\frac{\diver X}{n}g_0= \Hess\phi-\frac{\Delta\phi}{n}g_0=:Q,\qquad \Delta\diver X=\Delta^2\phi=:\psi.
\]
The identity becomes
\begin{equation}\label{eq:phi-identity}
 \frac A q\int_{\T^n}|\nabla\phi|^2u^q\dd
 =\int_{\T^n}Q(\nabla u,\nabla u)\dd
 -\frac1{2q}\int_{\T^n}\psi u^2\dd.
\end{equation}
On $B_r$ one has $\Hess\phi=2g_0$, $\Delta\phi=2n$, $\psi=0$, and $|\nabla\phi|^2=4|x|^2$. Therefore $Q$ and $\psi$ vanish on $B_r$, and
\eqref{eq:phi-identity} implies
\begin{equation}\label{eq: localize Pohozaev inequality}
\frac{4A}{q} \int_{B_r}|x|^2 u^qdx\le \int_{\T^n\backslash B_r}\left(Q(\nabla u,\nabla u)-\frac{1}{2q}\psi u^2\right)dx.
\end{equation}

Define
\begin{equation*}
 W=\int_{B_r}|x|^2u^q\dd,\qquad P=\int_{\T^n}u^p\dd,\qquad b=\av_{\T^n}u,\qquad
 F=f_Au^p,
\end{equation*}
where $\av$ denotes the average. 
The following two  lemmas will  dominate  the two sides of \eqref{eq: localize Pohozaev inequality} by $P$ from above and below respectively.

\begin{lemma}[Mass and exterior estimates]\label{lem:exterior}
We have
\begin{equation}\label{eq:mass}
 \int_{\T^n}F\dd=0,\qquad A\int_{\T^n}\phi u^p\dd=P,
 \qquad\|F\|_{L^1(\T^n)}\leq2P.
\end{equation}
For all sufficiently large $A$, one has
\begin{align}
 u&\leq C(P/A)^{1/p}
       &&\text{on }\T^n\setminus B_{r/4},\label{eq:outer-value}\\
 |F|&\leq CP
       &&\text{on }\T^n\setminus B_{r/4},\label{eq:outer-source}\\
 |u-b|+|\nabla u|&\leq CP
       &&\text{on }\T^n\setminus B_r,\label{eq:outer-oscillation}\\
 b&\leq C(P/A)^{1/p}+CP.\label{eq:average-bound}
\end{align}
\end{lemma}

\begin{proof}
Integrating both sides of \eqref{eq:main-pde}, we get the first two identities in \eqref{eq:mass}. So, $f_A=1-A\phi$ gives  $\|F\|_{L^1(\T^n)}\le P+A\int_{\T^n}\phi u^p= 2P$,  the last bound in \eqref{eq:mass}.

Choose a fixed radius $d>0$, which is smaller than $r/16$ and the injectivity radius. For every $x\notin B_{r/4}$, the ball $B_d(x)$ is Euclidean and lies outside $B_{r/8}$. Let $c_0=\min_{\T^n\setminus B_{r/8}}\phi>0$. For $A\geq2/c_0$, equation \eqref{eq:main-pde} gives
\[
 \Delta u=(A\phi-1)u^p\geq0\quad\text{on }B_d(x).
\]
The mean-value inequality for subharmonic functions and H\"older's inequality imply
\[
 u(x)\leq\av_{B_d(x)}u
 \leq\bigl(\av_{B_d(x)}u^p\bigr)^{1/p}
 \leq\left(\frac{P}{A c_0|B_d|}\right)^{1/p}.
\]
Here the last step uses $A\int\phi u^p=P$. Since $d$ is fixed, this proves  \eqref{eq:outer-value}. Since $|f_A|\leq CA$ for $A\geq1$, this also proves \eqref{eq:outer-source}.

Let $G(x,y)$ be the mean-zero Green kernel for $-\Delta$ on the torus. The kernels $G$ and $\nabla_xG$ are integrable in $y$, uniformly in $x$, and are bounded on fixed sets separated from the diagonal. Since $\int F=0$,
\begin{equation}\label{eq:representation}
 u(x)-b=\int_{\T^n}G(x,y)F(y) d y.
\end{equation}
For $x\notin B_r$, split the integral into $y\in B_{r/2}$ and $y\notin B_{r/2}$. In the first part, $d(x,y)\geq r/2$, so boundedness of the kernel and \eqref{eq:mass} give a bound $CP$. In the second part, use \eqref{eq:outer-source} and kernel integrability. The same split applies to the differentiated representation with $\nabla_xG$. This proves \eqref{eq:outer-oscillation}. Finally, evaluating at one fixed point outside $B_r$, \eqref{eq:outer-value} and \eqref{eq:outer-oscillation} give \eqref{eq:average-bound}.
\end{proof}

\begin{lemma}[Weighted mass estimate]\label{lem:weighted}
For $n\geq5$ and $A$ sufficiently large,
\begin{equation}\label{eq:weighted-mass}
 P\leq CW^{p/q}.
\end{equation}
\end{lemma}

\begin{proof}
Let $c=\min_{\T^n\setminus B_r}\phi>0$. By \eqref{eq:mass},
\[
 \int_{\T^n\setminus B_r}u^p\dd\leq\frac{P}{Ac}.
\]
For $A\geq2/c$, from $P=\int_{\T^n\setminus B_r}u^p\dd+\int_{B_r}u^p\dd$, it follows that $P\leq2\int_{B_r}u^p\dd$. As $q-p=1$, H\"older's inequality with exponents $q/p$ and $q$ gives
\begin{align*}
 \int_{B_r}u^p\dd
 =\int_{B_r}(|x|^2u^q)^{p/q}|x|^{-2p/q}\dd\leq W^{p/q}\left(\int_{B_r}|x|^{-2p}\dd\right)^{1/q}.
\end{align*}
The last integral  is finite if and only if $n>2p$. With  $p=\frac{n+2}{n-2}$, this condition is equivalent to $n^2-4n-4>0$, which holds for every integer $n\geq5$. This proves \eqref{eq:weighted-mass}.
\end{proof}

\section{Proof of the nonexistence theorem}\label{sec:proof}

\begin{proof}[Proof of Theorem~\ref{thm:main}]
 Take $A$ large enough for the preceding estimates to hold.  
First obtain a coarse bound for $W$. Choose a fixed smooth cutoff $\eta$ equal to zero on $B_{r/4}$ and one outside $B_{r/2}$. Its support lies in $\{f_A<0\}$ for all sufficiently large $A$. Testing the equation by $\eta^2u$ gives
\[
 \int\eta^2|\nabla u|^2\dd
 \leq2\left|\int\eta u\nabla\eta\cdot\nabla u\dd\right|
 \leq\frac12\int\eta^2|\nabla u|^2\dd
     +2\int|\nabla\eta|^2u^2\dd.
\]
Using \eqref{eq:outer-value} on the support of $\nabla\eta$  and $\eta$ yields
\begin{equation*}
 \int\eta^2\left(|\nabla u|^2+u^2\right)\dd\leq C(P/A)^{2/p}.
\end{equation*}
Since  $\eta=1$ on $\T^n\backslash B_r$, \eqref{eq: localize Pohozaev inequality} implies
\begin{equation*}
 AW\leq C(P/A)^{2/p}.
\end{equation*}
Applying Lemma~\ref{lem:weighted} and  $1-2/q=2/n$, this gives
\begin{equation}\label{eq:coarse}
 W^{2/n}\leq CA^{-1-2/p}.
\end{equation}

Next refine the estimate in \eqref{eq:phi-identity} by removing the constant part of $u$. By \eqref{eq:outer-oscillation},
\[
 \left|\int Q(\nabla u,\nabla u)\dd\right|\leq CP^2.
\]
Since the torus is closed, $\int\psi\dd=\int\Delta^2\phi\dd=0$. Thus
\begin{align*}
 \left|\int\psi u^2\dd\right|
 =\left|\int\psi\bigl[(u-b)^2+2b(u-b)\bigr]\dd\right|\leq C(P^2+bP),
\end{align*}
where \eqref{eq:outer-oscillation} is applied on the support of $\psi$. It follows from \eqref{eq:phi-identity} and \eqref{eq:average-bound} that
\begin{equation*}
 AW\leq C(P^2+bP)
 \leq CP^2+CA^{-1/p}P^{1+1/p}.
\end{equation*}
Using \eqref{eq:weighted-mass} and $1+1/p=q/p$, this gives
\[
 AW\leq CW^{2p/q}+CA^{-1/p}W.
\]
Divide by $W>0$ and use $2p/q-1=2/n$, followed by \eqref{eq:coarse}. We  obtain
\begin{equation}\label{eq:contradiction}
 A\leq CW^{2/n}+CA^{-1/p}
 \leq CA^{-1-2/p}+CA^{-1/p},
\end{equation}
A contradiction to the existence of a positive solution exists for all sufficiently large $A$.

Finally, $f_A(x_0)=1$ and $\phi$ is strictly positive away from $x_0$, so $f_A$ changes sign and
\[
 \int_{\T^n}f_A\dd=1-A\int_{\T^n}\phi\dd<0
\]
for all large $A$. 
\end{proof}

\begin{proof}[Proof of Corollary~\ref{cor:geometry}]
Choose an admissible $\phi$. If  $\widetilde g=w^{4/(n-2)}g_0$ had scalar curvature $f_A$, then $-c_n\Delta w=f_Aw^p$ and $u=c_n^{-1/(p-1)}w$ would  solve \eqref{eq:main-pde}, contradicting Theorem~\ref{thm:main}.
\end{proof}

\begin{remark}
The exact form $\phi=|x|^2$ near its zero is used to make both $Q$ and $\Delta^2\phi$ vanish in the region where a hypothetical solution could concentrate. The argument does not replace this assumption by an arbitrary nondegenerate quadratic minimum. In particular, an anisotropic quadratic form need not have vanishing trace-free Hessian.
\end{remark}

\begin{remark}
The only restriction to $n\geq5$ occurs in Lemma~\ref{lem:weighted}. In dimension five, $n-2p=1/3>0$, so the weighted integral is finite. It diverges in dimensions three and four. Thus the argument makes no nonexistence assertion in those dimensions and is compatible with the result of Escobar and Schoen \cite{ES86}.
\end{remark}

\end{document}